\documentclass{article}
\usepackage[utf8]{inputenc}
\usepackage{amsmath}
\usepackage{amssymb}
\usepackage{amsfonts}
\usepackage{amsthm}
\usepackage{tikz-cd}
\usepackage{graphicx}
\usepackage{comment}
\usepackage[maxbibnames=99, minbibnames=99, style=alphabetic]{biblatex}
\usepackage[margin=1.25in, headheight=1in]{geometry}
\graphicspath{ {images/} }
\title{Homology of loops on a splitting-rank symmetric space via the Geometric Satake for real groups}
\author{John M. O'Brien}

\newcommand{\oh}{\mathcal{O}}

\newcommand{\GL}{\textnormal{GL}}
\newcommand{\SL}{\textnormal{SL}}

\newcommand{\Sp}{\textnormal{Sp}}

\newcommand{\frakg}{\mathfrak{g}}
\newtheorem{lemma}{Lemma}
\def\Ad{\textnormal{Ad}}

\def\Mod{{\textnormal{Mod}}}
\def\gl{{\mathfrak{gl}}}
\def\pgl{{\mathfrak{pgl}}}
\def\sl{{\mathfrak{sl}}}
\def\so{{\mathfrak{so}}}

\def\PGL{{\textnormal{PGL}}}
\def\PSL{{\textnormal{PSL}}}
\def\PSO{{\textnormal{PSO}}}

\def\SO{{\textnormal{SO}}}
\def\O{{\textnormal{O}}}

\def\bbC{{\mathbb{C}}}
\def\bbF{\mathbb{F}}
\def\bbG{\mathbb{G}}
\def\bbH{{\mathbb{H}}}

\def\bbK{{\mathbb{K}}}

\def\bbO{{\mathbb{O}}}
\def\bbo{{\mathbb{O}}}

\def\bbQ{{\mathbb{Q}}}
\def\bbQ{{\mathbb{Q}}}
\def\bbR{{\mathbb{R}}}

\def\bbZ{{\mathbb{Z}}}

\def\calD{{\mathcal D}}

\def\calF{{\mathcal F}}
\def\calG{{\mathcal G}}

\def\calL{{\mathcal L}}
\def\calM{\mathcal M}

\def\calG{{\mathcal G}}
\def\calI{{\mathcal I}}
\def\frakc{{\mathfrak c}}
\def\frakf{{\mathfrak f}}
\def\frakh{{\mathfrak h}}
\def\frakm{{\mathfrak m}}
\def\frakt{{\mathfrak t}}

\def\boldQ{{\mathbf Q}}
\def\Fil{{\textnormal{Fil}}}

\def\Aut{{\mathop{\rm Aut\,}}}
  
\def\im{{\mathop{\rm im}}}

\def\Vect{{\mathop{\rm Vect}}}

\def\Gr{{\textnormal{Gr}}}
\def\Sp{{\textnormal{Sp}}}
\def\Spec{{\textnormal{Spec }}}
\begin{document}

\maketitle

\section*{Introduction}

Let $G$ be a complex reductive group with compact form $G_c$.  The Geometric Langlands Correspondence predicts a relationship between sheaves on the loop space $\Omega G_c$ and data attached to the dual reductive group $G^\vee$.  Central to the equivalence, then, is an understanding of the homology and equivariant homology of $\Omega G_c$ in terms compatible with the Langlands program. Ginzburg gave the first such description \cite{ginzburg2000perverse}, followed later by Bezrukavnikov, Finkelburg, and Mirković \cite{bezrukavnikov2014equivariant} and Yun and Zhu \cite{yun_integral_2011}. Here we discuss Yun and Zhu's formulation of $H_*(\Omega(G_c))$ in terms of the Langlands dual group:

\newtheorem*{t1}{Theorem 1 \cite{yun_integral_2011}}
\begin{t1}
For $G$ a connected reductive group over $\bbC$ with almost simple derived group, there is an isomorphism of Hopf algebras

$$
H_*(\Omega G_c, \bbZ[1/\ell_G]) \cong \oh(G^\vee_e)[1/\ell_G]
$$

where $\ell_G$ denotes the square of the ratio of root lengths of $G_c$, $e$ a regular nilpotent of $G^\vee$, and $G^\vee_e$ the centralizer of $e$ in $G^\vee$.
\end{t1}

They also prove a description of the equivariant cohomology in terms of the regular centralizer group scheme associated to $G^\vee$. 

The proof uses the Geometric Satake Correspondence and known homotopy equivalence $\Omega G_c \rightarrow \Gr = G(\bbC((t)))/G(\bbC[[t]])$ to canonically construct such a morphism of Hopf algebras.

The Relative Langlands program, a vast generalization of the Geometric Langlands program, studies sheaves on a wider range of spaces.  One interesting class of such spaces are the loop spaces of symmetric spaces.  Recall that a symmetric space $X_c = G_c/K_c$ is the quotient of $G_c$ by the fixed points of an involution.  Examples include spheres, projective spaces, and many other familiar spaces.

It is natural to ask whether a description of $H_*(\Omega(G_c/K_c))$ analogous to \cite{yun_integral_2011} holds.  There is a dual group--the Gaitsgory-Nadler dual group $G^\vee_X$--attached to such spaces.  It is a subgroup of $G^\vee$ with root data a modification of the relative roots of $G$.

However, describing $H_*(\Omega(G_c/K_c))$ in general is difficult even without involving dual groups--proofs involve passing to $\bbF_2$-coefficients to deal with non-orientability and seeing what information about $\bbZ$-coefficients we can glean. 

In principal, one could prove analogues of Theorem 1 by studying the algebraic loop space $\calL X = X(\bbC((t)))$.  However, sheaves on $\calL X$ are unwieldy, due to the tendency of $G(\bbC[[t]])$-orbits to be infinite dimensional and infinite codimensional. Instead, we use Quillen's result that  $\Omega(G_c/K_c)$ is homotopy equivalent to $\Gr_\bbR$, where $\Gr_\bbR$ is the affine Grassmannian $G_\bbR(\bbR((t)))/G_\bbR(\bbR[[t]])$ of the real form $G_\bbR$ of $G$ associated to $G_c/K_c$.  Here, the equivalent strata--the $G_\bbR(\bbR[[t]])$-orbits--are miraculously finite-dimensional and concrete.

In this paper we use this real analytic model to compute the homology and equivariant homology of a particularly nice family of loop spaces of symmetric spaces--the family of splitting-rank symmetric spaces.  This family is known to have no $2$-torsion in homology.  Furthermore, aside from the compact Lie groups, splitting-rank symmetric spaces dual group $G^\vee_X$ of type $A_{m-1}$.  If we choose centerless forms of the associated real groups, we get $G^\vee_X \cong \SL(m, \bbC)$.

The results of this paper will feature in a couple upcoming works. 
 With Tsao-Hsien Chen, Mark Macerato, and David Nadler we will prove a version of the Derived Satake Equivalence for $\GL_n(\bbH)$ \cite{chen2022quaternionic}.  Later, Chen and O'Brien will use this description to generalize the Derived Satake Equivalence to the remaining real splitting-rank groups. 

\newtheorem*{t2}{Theorem 2}
\begin{t2}
For $X= G_c/K_c$ a compact, splitting-rank symmetric space of adjoint type, there is an isomorphism of Hopf algebras

$$
H_*(\Omega(X), \bbZ[1/\ell_X]) \cong \oh(G^\vee_{X,e}),
$$

where $e$ is a regular nilpotent in $\frakg^\vee_{X,e}$, $G^\vee_{X,e}$ the centralizer of $e$, and $\ell_X$ the product of a few small primes.
\end{t2}

The equivariant calculation also holds, relating the equivariant homology to the regular centralizer of $G^\vee_X$.

\section*{Factorization and homology}
Here we endow $H_*(\Omega(X))$ with the structure of a Hopf algebra in an \textit{a priori} different way than loop concatenation, using factorization for the real affine Grassmannian.  Recall that, for $\Gamma$ a smooth complex curve with real form $\Gamma_\bbR$, we have the Beilinson-Drinfeld Grassmannian:

$$
\Gr_{\bbR, \Gamma_\bbR^2} = \{(x,y,\varepsilon, \tau): x,y\in \Gamma_\bbR;\; \varepsilon: E\rightarrow \Gamma_\bbR;\; \tau \textnormal{ a trivialization of E away from (x,y)})\}
$$

Here $E$ is a principal $G_\bbR$-bundle over $\Gamma_\bbR$.

This is a real analytic space over $\Gamma_\bbR^2$.  Its fibers over the diagonal embedding of $\Gamma_\bbR$ are given by $\Gr_{\bbR}$, while fibers elsewhere are isomorphic to $\Gr_{\bbR} \times \Gr_{\bbR}$. Specialization of the constant sheaf in this family gives us a product $\wedge_{sp}: H_*(\Gr_{\bbR} \times \Gr_{\bbR}) \rightarrow H_*(\Gr_{\bbR})$.  In the case of complex groups, this product is identical to the ordinary Pontryagin product by \cite{yun_integral_2011}. Dually, we denote the coproduct $\Delta_{sp}: H^*(\Gr_\bbR) \rightarrow H^*(\Gr_\bbR \times \Gr_\bbR)$.

Factorization interacts with the Geometric Satake Equivalence for real groups in the following way.  Consider the subcategory $\boldQ(\Gr_\bbR)$ of $P_{G_\bbR(\bbR[[t]])}(\Gr_\bbR)$ defined by Nadler for real groups \cite{nadler_perverse_2005}.  For splitting-rank cases, this is the entire category $P_{G_\bbR(\bbR[[t]])}(\Gr_\bbR)$ .  It is a Tannakian category with fiber functor $H^*: \boldQ(\Gr_\bbR) \rightarrow \Vect_\bbZ$ and monoidal structure given by convolution, or equivalently (up to sign), by fusion.  We use the latter: for $\calF, \calG \in \boldQ(\Gr_\bbR)$, let $\calF_{\Gamma_\bbR}$ be the "globalization" on $\Gr_{\bbR, \Gamma_{\bbR}}$.  Let $\Delta: \Gamma_\bbR \rightarrow \Gamma_\bbR^2$ denote the inclusion of the diagonal and $i: \Gamma_\bbR^2 - \Delta \rightarrow \Gamma_\bbR^2$ the inclusion of the complement of the diagonal.  Then

$$
\calF * \calG := \Delta^*i_{!*}(\calF_{\Gamma_\bbR} \boxtimes \calG_{\Gamma_\bbR})|_{\Gamma^2_\bbR-\Delta}
$$

Using factorization, we have the following lemma concerning cohomology classes:

\begin{lemma}\label{tensorend}
We have isomorphisms, for $\calF,\calG\in Q(\Gr_\bbR)$ and $h\in H^*(\Gr_\bbR, k)$,

$$
\begin{tikzcd}
H^*(\calF)\otimes H^*(\calG)\arrow[d, "\cup(\Delta_{sp} h)"]\arrow[r, "\sim" ] & H^*(\calF*\calG) \arrow[d, "\cup h"]  \\
H^*(\calF)\otimes H^*(\calG) \arrow[r, "\sim"] & H^*(\calF*\calG) 
\end{tikzcd}
$$

An identical statement holds for ${T_\bbR}$-equivariant homology, for ${T_\bbR}$ the maximal compact torus of $G_\bbR$.
\end{lemma}

We will use this lemma in several places to construct morphisms into the dual group and points in the Lie algebra.

\section*{MV Filtration}
As in \cite{yun_integral_2011} we collect some results on an equivariant version of the coweight filtration of Mirkovic-Vilonen.  

Let $T_\bbR$ denote the maximal torus of the maximal compact group $K_\bbR$ and $S$ the complexification of the maximal split torus $S_\bbR$.  Let $S_{\mu, \bbR}$ be the $U_\bbR((t))$ orbit through $t^\nu$ for $\nu \in \Lambda_{S}$, where $\Lambda_S$ is the lattice of coweights of the split torus $S$.  

By Nadler's Theorem 8.5.1 \cite{nadler_perverse_2005}, for $\calF \in \boldQ(\Gr_\bbR)_{^pH}$ we have that $H^*_c(S_{\nu, \bbR}, \calF)$ is concentrated in degree $k = \langle \check{\rho},\nu\rangle$.  Since every $\nu = \theta(\mu) + \mu$ for some complex cocharacter $\mu$ and $\theta$ the involution induced by complex conjugation, we have that $k$ is of even parity.

The MV filtration filters $H^*(\calF)$ by coweights, with $\Fil_{\geq \nu} H^*(\calF) = \ker(H^*(\calF) \rightarrow H^*(S_{<\nu, \bbR}, \calF)$, where $S_{<\nu, \bbR} = \overline{S_{\nu, \bbR}} - S_{\mu,\bbR}$.  Likewise, we may filter $H^*_{T_\bbR}(\calF)$ by $\Fil_{\geq \nu}^{{T_\bbR}} H_{T_\bbR}^*(\calF) = \ker(H_{T_\bbR}^*(\calF) \rightarrow H_{T_\bbR}^*(S_{<\nu, \bbR}, \calF)$.

For $R_{T_\bbR} = H_{T_\bbR}^*(*, \bbZ)$, we have the following theorem:

\newtheorem*{th3}{Theorem 3}
\begin{th3}
There is a natural isomorphism

$$
H_{T_\bbR}^*(\Gr_\bbR, -) \cong H^*(\Gr_\bbR, -) \otimes R_{T_\bbR}: \boldQ(\Gr_\bbR)_{^pH} \rightarrow \Mod^*(R_{T_\bbR})$$

\end{th3}

\begin{proof}

The proof follows exactly as in \cite{yun_integral_2011}. Since $H^*_c(S_{\nu, \bbR}, \calF)$ is concentrated in a single degree, applying a spectral sequence gives us that $H^*_{T_\bbR,c}(S_{\nu, \bbR},\calF) \cong H^*_c(S_{\nu, \bbR}, \calF) \otimes R_{T_\bbR}$. 

Then $H^*_{T_\bbR}(\calF)$ has a spectral sequence with $E_1$ terms $H^*_{{T_\bbR},c}(S_{\nu, \bbR},\calF)$.  By the vanishing along parity, this degenerates.  

By degree consideration, since $\Fil^{T_\bbR}_{>\mu} H^*(\calF)$ is concentrated in degree above $\langle \check{\rho}, \nu \rangle$, then the following short exact sequence splits canonically:

$$
0\rightarrow \Fil^{T_\bbR}_{>\mu} H_{T_\bbR}^*(\calF) \rightarrow \Fil^{T_\bbR}_{\geq \mu} H_{T_\bbR}^*(\calF) \rightarrow H^*_{T_\bbR,c}(S_{\nu, \bbR}, \calF) \rightarrow 0
$$

Hence we have the following isomorphisms:

$$
H_{T_\bbR}^*(\calF) \cong \bigoplus_{\mu \in \Lambda_S} H^*_{{T_\bbR},c}(S_{\nu, \bbR}, \calF) \cong \bigoplus_{\mu \in \Lambda_S} H^*_{c}(S_{\nu, \bbR},\calF) \otimes R_{T_\bbR} \cong H^*(\calF) \otimes R_{T_\bbR}
$$
\end{proof}

We observe that this result can be strengthened.  Let $M_\bbR$ denote the Levi factor of the minimal parabolic $P_\bbR \subset G_\bbR$.  We have that $M_\bbR$ commutes with every $t^\nu$, hence we get a filtration on the $M_\bbR$-equivariant cohomology.

Taking invariants of the Weyl group of the Levi factor $W_M$, we get graded pieces $H^*_{M_\bbR,c}(S_{\nu, \bbR},\calF) \cong (H^*_c(S_{\nu, \bbR}, \calF) \otimes R_{T_\bbR})^{W_M}$.  By degree considerations, as above, the filtration splits. Hence,

$$
H_{M_\bbR}^*(\calF) \simeq \bigoplus_{\mu \in \Lambda_S} H^*_{{M_\bbR},c}(S_{\nu, \bbR}, \calF) \cong \bigoplus_{\mu \in \Lambda_S} (H^*_c(S_{\nu, \bbR}, \calF) \otimes R_{T_\bbR})^{W_M}
$$

Thus, we have an analogous natural isomorphism:

\newtheorem*{t4}{Theorem 4}
\begin{t4}
There is a natural isomorphism

$$
H_{M_\bbR}^*(\Gr_\bbR, -) \cong (H^*(\Gr_\bbR, -) \otimes R_{T_\bbR})^{W_M}: \boldQ(\Gr_\bbR)_{^pH} \rightarrow \Mod^*(R_{M_\bbR})$$

\end{t4}

\section*{The canonical automorphism}
In this section, we use Lemma 2 to construct a canonical automorphism of 

$$H_{T_\bbR}^*(-) \otimes_{R_{T_\bbR}} H^{T_\bbR}_*(\Gr_\bbR): \boldQ(\Gr_\bbR)_{^pH} \rightarrow \Mod^*(H^{T_\bbR}_*(\Gr_\bbR))$$

and likewise for their ${M_\bbR}$-equivariant counterparts. We use an adaptation of the approach in \cite{yun_integral_2011}.

Recall that, by the natural cell structure on $\Gr_\bbR$, $H_{T_\bbR}^*(\Gr_\bbR)$ and $H_*^{T_\bbR}(\Gr_\bbR)$ are free $R_{T_\bbR}$-modules concentrated in even degree.  For dual bases $h^i$ and $h_i$ of $H_{T_\bbR}^*(\Gr_\bbR)$ and $H_*^{T_\bbR}(\Gr_\bbR)$, we set

$$
\sigma_{T_\bbR}(v\otimes h) = \sum_i (h^i\cup h)\otimes (h_i \wedge_{sp} h)
$$

for $\wedge$ the product on homology induced by specialization.  The construction is independent of choice of basis.

Likewise, $H_{M_\bbR}^*(\Gr_\bbR)$ and $H^{M_\bbR}_*(\Gr_\bbR)$ are free $R_{M_\bbR}$-modules concentrated in even degree, as $W_M$-fixed points of the action on $H_{T_\bbR}^*(\Gr_\bbR)$ and $H_*^{T_\bbR}(\Gr_\bbR)$.  Hence we may define $\sigma_{M_\bbR}$ analogously.

We claim that $\sigma_{T_\bbR}$ induces a morphism of Tannakian group schemes

$$\sigma_{T_\bbR}: \Spec H_*^{T_\bbR}(\Gr_\bbR) \rightarrow \Aut^\otimes H_{T_\bbR}^* \cong G^\vee_X \times \Spec R_{T_\bbR}$$

where $G^\vee_X$ is the Gaitsgory-Nadler dual group of $X$.  Two properties need to be proven for this to hold.  First, we need that $\sigma_{T_\bbR}$ is a tensor automorphism so that it represents a $H_*^{T_\bbR}(\Gr_\bbR)$-point of $\Aut^\otimes H_{T_\bbR}^*$.  Then, we need that that the induced map respects multiplications in each algebraic group.  Both properties essentially follow from adjointness of multiplication and comultiplication in  $H_*^{T_\bbR}(\Gr_\bbR)$ and $H^*_{T_\bbR}(\Gr_\bbR)$.  We will give the details below.

\begin{lemma}
The automorphism $\sigma_{T_\bbR}$ of $H_{T_\bbR}^*(-) \otimes_{R_{T_\bbR}} H^{T_\bbR}_*(\Gr_\bbR)$ is a tensor automorphism.
\end{lemma}
\begin{proof}
This follows from adjointness of $\Delta_{sp}$ and $\wedge_{sp}$.

Let $v_1$ and $v_2$ be elements of $H_{T_\bbR}^*(\calF_1)$ and $H_{T_\bbR}^*(\calF_2)$.  Under the isomorphism $H_{T_\bbR}^*(\calF_1 * \calF_2) \cong H_{T_\bbR}^*(\calF_1) \otimes_{R_{T_\bbR}} H_{T_\bbR}^*(\calF_2)$, $v_1\otimes v_2$ identifies with an element of $H_{T_\bbR}^*(\calF_1 * \calF_2)$.

Then we have the following computation:

$$
\sigma_{T_\bbR}(v_1\otimes 1) \otimes \sigma_{T_\bbR}(v_2\otimes 1) = \sum_{i,j} (h^i\cup v_1)\otimes (h^j \cup v_2) \otimes (h_i \wedge_{sp} h_j)
$$

In other words, for $\calM$ the matrix of $\wedge_{sp}: H_*^{T_\bbR}(\Gr_\bbR) \otimes_{R_{T_\bbR}} H_*^{T_\bbR}(\Gr_\bbR) \rightarrow H_*^{T_\bbR}(\Gr_\bbR)$ with respect to the bases $\{h_i \otimes h_j\}$ and $\{h_k\}$, we have that 

$$
\sigma_{T_\bbR}(v_1\otimes 1) \otimes \sigma_{T_\bbR}(v_2\otimes 1) = \sum_{i,j, k} \calM_{k, ij} (h^i\otimes h^j) \cdot (v_1 \otimes v_2) \otimes h_k.
$$

where $\cdot$ represents the factor-wise cup product action.

Now, we use Lemma 1 on cupping with cohomology classes to calculate $\sigma_{T_\bbR}((v_1\otimes v_2) \otimes 1)$.

We have that

$$
\sigma_{T_\bbR}((v_1\otimes v_2) \otimes 1) = \sum_k \Delta_{sp}(h^k)(v_1\otimes v_2)\otimes h_k.
$$

For $\calD$ the matrix of $\Delta_{sp}: H^*_{T_\bbR}(\Gr_\bbR) \rightarrow H^*_{T_\bbR}(\Gr_\bbR) \rightarrow H^*_{T_\bbR}(\Gr_\bbR)$, we have that 

$$
\sigma_{T_\bbR}((v_1\otimes v_2) \otimes 1) = \sum_{i,j, k} \calD_{ij, k} (h^i\otimes h^j) \cdot (v_1 \otimes v_2) \otimes h_k.
$$

Since $\Delta_{sp}$ is the adjoint of $\wedge_{sp}$, we have that $\calD = \calM^T$.  Hence, we have our tensor structure:

$$
\sigma_{T_\bbR}((v_1\otimes v_2) \otimes 1) = \sigma_{T_\bbR}(v_1\otimes 1) \otimes \sigma_{T_\bbR}(v_2\otimes 1)
$$
\end{proof}

\begin{lemma}
The map $\sigma_{T_\bbR}: \Spec H^{T_\bbR}_*(\Gr_\bbR) \rightarrow G^\vee_X \times \Spec R_{T_\bbR}$ is a morphism of algebraic groups.
\end{lemma}

\begin{proof}
This follows from adjointness of $\cup$ and $\Delta_*$, where $\Delta_*$ is the coproduct in homology.

We need to check that the algebraic group multiplication $\Spec(\Delta_*)$ of $H^{T_\bbR}_*(\Gr_\bbR)$ is mapped to the multiplication $\mu$ of $G^\vee_X$.  This corresponds to the following diagram of categories commuting

$$
\begin{tikzcd}
\boldQ(\Gr_\bbR)_{^pH} \arrow[d, "\mu*"] \arrow[r,"\sigma_{T_\bbR}"] &\Mod^*(H_*^{T_\bbR}(\Gr_\bbR)) \arrow[d, "\Delta_*"]\\
\boldQ(\Gr_\bbR \times \Gr_\bbR)_{^pH} \arrow[r, "\sigma_{T_\bbR} \otimes \sigma_{T_\bbR}"] &\Mod^*(H_*^{T_\bbR}(\Gr_\bbR)\otimes H_*^{T_\bbR}(\Gr_\bbR))
\end{tikzcd}
$$

Here, $\mu^*$ is the pullback of the multiplication map $\mu: G_\bbR(\bbR((t))) \times_{G(\bbR[[t]])} \Gr_{\bbR} \rightarrow \Gr_\bbR$.  

Now, $(\sigma_{T_\bbR}\otimes\sigma_{T_\bbR}) \circ \mu^*$ maps a sheaf $\calF$ to  $H^*_{T_\bbR}(\calF)\otimes H_*^{T_\bbR}(\Gr_\bbR) \otimes H_*^{T_\bbR}(\Gr_\bbR)$ together with automorphism

$$
v\otimes 1 \otimes 1 \longmapsto \sum_{i,j} (h^i\cup h^j) \cup v \otimes (h_i\otimes h_j).
$$

Similarly, $\Delta_*\circ \sigma_{T_\bbR}$ maps $\calF$ to its $H^*_{T_\bbR}(\calF)\otimes H_*^{T_\bbR}(\Gr_\bbR) \otimes H_*^{T_\bbR}(\Gr_\bbR)$ together with automorphism

$$
v\otimes 1 \otimes 1 \longmapsto \sum_k h^k\cup v \otimes \Delta_*(h^k).
$$

By an adjointness argument identical to that for the previous lemma, the automorphisms are equal.
\end{proof}
Since the morphism respects the MV filtration, its image lies in a Borel subgroup of $G^\vee_X \times \Spec R_{T_\bbR}$.  Hence, we write

$$
\sigma_{T_\bbR}:  \Spec H_*^{T_\bbR}(\Gr_\bbR) \rightarrow B^\vee_X \times \Spec R_{T_\bbR}
$$

Likewise, in the ${M_\bbR}$-equivariant setting, we have a morphism of Tannakian groups

$$
\sigma_M: \Spec H_*^{M_\bbR}(\Gr_\bbR) \rightarrow B^\vee_X \times \Spec R_{M_\bbR}.
$$

De-equivariantizing, we have a base morphism in ordinary homology

$$
\sigma: \Spec H_*(\Gr_\bbR) \rightarrow B^\vee_X
$$

We make an observation about the nature of this morphism, to be used in a later paper. Analogously to Remark 3.4 in \cite{yun_integral_2011}, we have the following result regarding the composition $\tau_{T_\bbR}= \pi \circ \sigma_{T_\bbR}: \Spec H_*^{T_\bbR}(\Gr_\bbR) \rightarrow T^\vee_X \times \Spec R_{T_\bbR}$ for $\pi$ the natural projection $B^\vee_X \times \Spec R_{T_\bbR} \rightarrow T^\vee_X \times \Spec R_{T_\bbR}$.

\begin{lemma}
    We have the following equality of morphisms:

 $$\tau_{T_\bbR} = \Spec(Loc_*): \Spec H_*^{T_\bbR}(\Gr_\bbR) \rightarrow T^\vee_X \times \Spec R_{T_\bbR}$$
\end{lemma}

\begin{proof}
    We follow the proof closely.  For each weight vector $v_\lambda \in H_{{T_\bbR},c}(S_{\lambda,\bbR, \calF)}$, we have that, for $i_\lambda$ the inclusion of $\{t_\lambda\}$ into $\Gr_\bbR$

    $$
    \sigma_{T_\bbR}(v_\lambda) = \sum_j i_\lambda^*(h^j) \cup v_\lambda \otimes h_j = v_\lambda \cup i_{\lambda,*}(1)
    $$

    Since equivariant localization $Loc^*$ is the product of all $i_\lambda^*$, its dual $Loc_*$ is the sum of all $i_{\lambda,*}$, the same effect of $\sigma_{T_\bbR}$.

    Hence  $\tau_{T_\bbR} = \Spec(Loc_*)$.
\end{proof}

\section*{Primitive classes}
We seek to show that image of $\sigma$ centralizes a principal nilpotent of $\frakg^\vee_X$.  For now, we consider the adjoint forms of $G_\bbR: \PSL(n,\bbH), \PSO(1,2n-1),$ and $PE_6(F_4)$ along with their corresponding symmetric spaces.

In every case, our relative dual group $G^\vee_X = \SL(n,\bbC)$.  We fix some notation.  Once and for all, identify $\Phi^\vee_X = \{e^i-e^j\}$ with simple roots $\alpha^\vee_i = e_i-e_{i+1}$.  

Let $\omega_i = (\alpha^\vee_i)^*$ denote the fundamental coweight dual to $\alpha^\vee_i$.  Likewise, let $\omega_{1,i}$ denote the $n$ coweights in the Weyl group orbit of $\omega_{1,i}$.  We have that $\omega_{1,1} = \omega_1$, $\omega_{1,i} = -\omega_1+\omega_i$ for $0<i<n$, and $-\omega_n$ for $i=n$.  Note that $\omega_{1,i}$ with $i<n$ give a basis for $\Lambda_S$, the lattice of coweights of $S$.  In terms of the standard basis $e^i$ of $\bbZ^n$, $\omega_{1,i} = \frac{1}{n}( (n-1)e_i - \sum_{j\neq i} e_j)$, identified with the appropriate diagonal matrices.


Let $\alpha^\vee_{1,i}$ with $i<n$ denote the dual basis of $\Lambda_S$ to $\omega_{1,i}$--i.e. $\alpha^\vee_{1,i} = \sum_{j\geq i} \alpha^\vee_j$.  In terms of our standard basis $e^i$, we have that $\alpha^\vee_{1,i} = e^i-e^n$

Consider the minuscule variety $\Gr_{\omega_1, \bbR}\cong G_\bbR/P_{\omega_1,\bbR}$, where $P_{\omega_1,\bbR}$ is the parabolic of $G_\bbR$ generated by the root subgroups $U_{\alpha^\vee}$ with $\langle \omega_1, \alpha^\vee\rangle\leq 0$.  When $G_\bbR = \PGL(n,\bbH), \PSO(1,2n-1),$ or $PE_6(F_4)$, we have that $\Gr_{\omega^1,\bbR} = \bbH P^{n-1}, S^{2n-2},$ or $\bbo P^2$ respectively. In all cases, $H^*(\Gr_{\omega^1,\bbR}) \cong \bbZ[e]/(e^n)$, where $e$ is of degree $d=4, 2n-2,$ or $8$ respectively.  

In the projective space cases, we may realize $e$ as the Euler class of the dual of the tautological bundle.  For the quaternionic case, recall that $\bbH P^{n-1}$ may be identified with the set of quaternionic lines in $\bbH^n$ --hence, it has a tautological bundle.  Likewise $\bbO P^1$ may be identified with the set of octonionic lines in $\bbO \times \bbO$--for details, see the discussion in \cite{baez01}. However, $\bbO P^2$ does not have this realization.

Hence, we use an alternative description for the octonionic tautological line bundle. Let $n=2,3$.  Realize $\bbO P^{n-1}$ and as the set of rank-one projections in the corresponding Jordan algebra $\mathfrak{h}_n(\bbO)$ of matrices Hermitian in $\bbO$ over $\bbR$.  Taking the image of each projection $p$ on $\bbO^n = \bbR^{8n}$ gives an 8-dimensional subspace of $\bbR^{8n}$. Let $\calL_{taut,\bbO p^{n-1}} = \{(p,v): p\in \bbO P^{n-1}, v\in \im p\}.$ denote the disjoint (as a set) union of these subspaces.  Under the map $\calL_{taut,\bbO P^{n-1}} \rightarrow \O P^{n-1}$ sending $(p,v) \mapsto v$, this is an 8-dimensional real vector bundle over $\bbO P^{n-1}$.  Observe that, for $\bbO P^{1}$ this corresponds naturally with the set of octonionic lines in $\bbO \times \bbO$.

For the even spheres that are not of the form $\bbK P^{n-1}$, there is no tautological line bundle. Instead, we use $-\frac{1}{2}e_1(TS^{2n-2})$.

Since $H_*(\Gr_{\omega^1, \bbR})$ generates the homology of $\Gr_\bbR$ by an argument of Bott \cite{Bott_1958}, we may extend the class $e\in H^*(\Gr_{\omega^1,\bbR})$ to a class $e \in H^*(\Gr_\bbR)$. We call this class the determinant class by analogy with the determinant line bundle over the complex affine Grassmannian.  Indeed, for $G_\bbR$ defined over $\bbH$ we may repeat the construction in \cite{yun_integral_2011} to get an honest quaternionic determinant line bundle over $\Gr_\bbR$.

Using \ref{tensorend} and the fact that $e$ is primitive, we get that taking the cup product with $e$ produces an endomorphism of $H^*: Q(\Gr_\bbR)_{^pH} \rightarrow \Vect_k$ with tensor formula $e \otimes 1 + 1\otimes e$.  In diagrams,

$$
\begin{tikzcd}
H^*(\calF)\otimes H^*(\calG)\arrow[d, "\cup(e \otimes 1 + 1 \otimes e)"]\arrow[r, "\sim" ] & H^*(\calF*\calG) \arrow[d, "\cup e"]  \\
H^*(\calF)\otimes H^*(\calG) \arrow[r, "\sim"] & H^*(\calF*\calG) 
\end{tikzcd}
$$

According to the Tannakian dictionary, such an endomorphism corresponds to an element of $e_X\in \frakg_X^\vee$.  According to the basis $\{1,\alpha, \alpha^2,\dots, \alpha^{n-1}\}$ of $H^*(\Gr_\bbR)$, cupping with $\alpha$ corresponds to the standard principal nilpotent in $\frakg_X^\vee$.

Now, we give the class $M_\bbR$-equivariant structure, giving rise to an element $e^M$ restricting to $e$.  As before, $e^M$ will be primitive in $H^*_{M_\bbR}(\Gr_\bbR,\bbQ)$.  Hence cupping with $e^M$ will correspond to an element of $\frakg_X^\vee\otimes R_{M_\bbR}$ with constant part $e$.  By degree considerations, the nonconstant part $p^M = e^M-e$ lies in $\frakt_X^\vee \otimes H^d_{{M_\bbR}_\bbR}(\Gr_\bbR)$.  Hence $p^M$ is of form $\sum_i g_i\otimes h_i$, where $g_i$ are elements of $\frakt_X^\vee$, identified with coweights $\nu_i \in \Lambda_S$.  Furthermore, $h_i$ is a polynomial function $\frakt \rightarrow \bbC$ invariant under $W_\bbR$. The action of $p^M$ on $H_{M_\bbR}^*(S_{\nu, \bbR})$ is given by $t^\nu \mapsto \sum_i \langle g_i, \nu \rangle h_i$.  

Recall that we may identify the fixed points of the $M_\bbR$-action on $\Gr_{\omega_1, \bbR}$ with $t^{\omega_{1,i}}$ for $i<n$.  Restricting to the fiber of our line bundle over $t^{\omega_{1,i}}$, we get that $\frakc \cap \frakm_\bbR$ acts by multiplication by a polynomial $f_i: \frakc \rightarrow \bbC$ invariant under the Weyl group of $M_\bbR$.   Using our dual basis $\alpha_{1,i}$ to $\omega_{1,i}$, we may calculate that $p^M = \sum_{i<n} \alpha^\vee_{1,i}\otimes f_i$.

Using this formula, we calculate $p^M$--and hence $e^M$--case-by-case by calculating $f_i$:

\subsection*{Orthogonal case}
For $G_\bbR = \PSO(1,2n-1)$, we use $e^M = \frac{1}{2}e^M_1(TS^{2n-2})$ with ${M_\bbR} = \SO(2n-2)$.  We may identify the restriction of $TS^{2n-2}$ to $\{t^{\omega_1}\}$ with the vector space $\so(2n-1)/\so(2n-2)$.  The left action of $\SO(2n-2)$ may be identified with the vector representation.  Hence the nonconstant portion of the Euler class is given by 

$$
p^M = \frac{1}{2}\alpha^\vee_1\otimes t_1t_2\dots t_n
$$

where $t_i$ are the coordinates of the copy of $\SO(2n-2)$ in the lower right corner of $\SO^+(1,2n-1)$.

In terms of matrices, then, we have that

$$
e^M
\begin{bmatrix}
    0 & t_1 & \\
    -t_1 & 0 & \\
     & & 0 & t_2 \\
     & & -t_2 & 0 \\
     & & & & \ddots \\
     & & & & & 0 & t_n \\
     & & & & & -t_n & 0
\end{bmatrix}
=
\begin{bmatrix}
    \frac{1}{2}t_1t_2\dots t_n & 1\\
    0 & -\frac{1}{2}t_1t_2\dots t_n 
\end{bmatrix}
$$

\subsection*{Octonionic cases}
For $G_\bbR = PE_6(F_4)$ or $\PSO(1,9)$, we may use the equivariant structure of the tautological line bundle on $\bbo P^n$ to calculate our class.  We treat the case of $\PSO(1,9)$ first:  here, ${M_\bbR} = \SO(8)$ and $\Gr_{\mu,\bbR} = \bbo P^1$.  The center of the double cover $\textnormal{Spin}(9)$ of the maximal compact subgroup $\SO(9)$ acts nontrivially on fibers.  Hence we calculate the equivariant class by working over $H^*_{\textnormal{Spin}(9)}(pt)$, then  pushing back to $H^*_{\SO(9)}(pt)$, picking up a coefficient of $\frac{1}{2}$.. Restricting to the north pole $\{t^{\omega_1}\}$, the bundle becomes the vector representation of $\mathfrak{m}$. Hence, for $t_i,$ $i=1$ to $4$ the coordinates of the Cartan of $\so(8)$, we get that

$$
p^M = \frac{1}{2}\alpha^\vee_1 \otimes t_1t_2t_3t_4
$$

In terms of matrices, then,

$$
e^M
\begin{bmatrix}
    0 & t_1 & \\
    -t_1 & 0 & \\
     & & 0 & t_2 \\
     & & -t_2 & 0 \\
     & & & & 0 & t_3 \\
     & & & &  -t_3 & 0 \\
     & & & & & & 0 & t_4 \\
     & & & & & & -t_4 & 0
\end{bmatrix}
= 
\begin{bmatrix}
    \frac{1}{2}t_1t_2t_3t_4 & 1\\
    0 & -\frac{1}{2}t_1t_2t_3t_4 
\end{bmatrix}
$$

Notice that this is the same formula as we calculated using the adjoint in the prior section.

For $G_\bbR = PE_6(F_4)$, we observe that ${M_\bbR} = \textnormal{Spin}(8)$.  Since the compact group $F_4$ is both simply-connected and centerless, we do not pick up any fractional coefficients.  Restrict first to $\{t^{\omega_{1,1}}\}$.  To calculate the class associated to the $M$-action on this fiber, consider first the orbit of the stabilizer $P_{\omega_{1,3}}$ of $\{t^{\omega_{1,3}}\}$ through  $\{t^{\omega_{1,1}}\}$.  This orbit is naturally identified with a copy of $\bbO P^1 \subset \bbO P^2$, corresponding to the upper left copy of $\frakh_2(\bbO) \subset \frakh_3(\bbO)$.  Moreover, ${M_\bbR} \subset P_{\omega_{1,3}}$, so we may use the calculation for $\PSO(1,9)$ to get the polynomial corresponding to the restriction to $\{t^{\omega_{1,1}}\}$--it is $p_0(x) = t_1t_2t_3t_4$.  

To get the restrictions to the other two fixed points, we may repeat this construction--however, we must be careful calculating which weights of the Cartan of $\frakm$ show up.  Instead we calculate by using automorphisms of $\bbO P^2$ that, in a sense, "permute the coordinates."  Recall that $\so(8)$ has a canonical outer triality permuting the coweights $\omega_{vect} = (1,0,0,0)$, $\omega_{spin-} = (\frac{1}{2}, \frac{1}{2}, \frac{1}{2}, -\frac{1}{2})$ and $\omega_{spin+} = (\frac{1}{2}, \frac{1}{2}, \frac{1}{2}, \frac{1}{2})$ and, dually, their associated representations of the dual group $V^{vect}$, $V^{spin-}$, and $V^{spin+}$.  Moreover, as an $\so(8)$-representation, $\frakf_4 \cong \so(8)\oplus V^{vect}\oplus V^{spin-} \oplus V^{spin+}$, extending the triality to an automorphism of $\frakf_4$ \cite{baez01}. The Weyl group of $\frakf_4$ induces the triality, in fact giving the decomposition $W(F_4) = W(\SO_8) \rtimes S_3$, where $S_3$ is the symmetric group on three letters, which we may identify with $\omega_{1,i}$.  

Let $\tau_{12}$ denote the a lift of the transposition swapping $\omega_{1,1}$ and $\omega_{1,2}$, fixing $\omega_{1,3}$.  We have that $\tau_{12}$ induces one of the triality automorphisms of $\so(8)$--without loss of generality, assume it swaps $\omega_{vect}$ and $\omega_{spin-}$, fixing $\omega_{spin+}$.  With respect to the basis $e_i$ of the weights of $\so(8)$, we have that $\tau_{12}$ has the following matrix $A_{12}$:

$$
A_{12} = \left(\begin{matrix}
\frac{1}{2} & \frac{1}{2} & \frac{1}{2} & \frac{-1}{2} \\
\frac{1}{2} & \frac{1}{2} & \frac{-1}{2} & \frac{1}{2} \\
\frac{1}{2} & \frac{-1}{2} & \frac{1}{2} & \frac{1}{2} \\
\frac{-1}{2} & \frac{1}{2} & \frac{1}{2} & \frac{1}{2}
\end{matrix}\right)
$$

Hence $\tau_{12}$ maps $t_1$ to $\frac{1}{2}t_1+\frac{1}{2}t_2 + \frac{1}{2}t_3 - \frac{1}{2}t_4$, and likewise for the other $t_i$.  Therefore we get that our polynomial for $t^{\omega_{1,2}}$ is given, up to sign, by 

$$
p_-(x) = (\frac{t_1}{2}+\frac{t_2}{2} + \frac{t_3}{2} - \frac{t_4}{2})(\frac{t_1}{2}+\frac{t_2}{2} - \frac{t_3}{2} + \frac{t_4}{2})(\frac{t_1}{2}- \frac{t_2}{2} + \frac{t_3}{2} + \frac{t_4}{2})(-\frac{t_1}{2}+\frac{t_2}{2} + \frac{t_3}{2} + \frac{t_4}{2}).
$$

By a similar argument featuring $\tau_{13}$, we can get that the polynomial at $t^{\omega_{1,3}}$ is given, up to sign, by

$$
p_+(x) = (\frac{t_1}{2}+\frac{t_2}{2} + \frac{t_3}{2} + \frac{t_4}{2})(\frac{t_1}{2}+\frac{t_2}{2} - \frac{t_3}{2} - \frac{t_4}{2})(\frac{t_1}{2}- \frac{t_2}{2} + \frac{t_3}{2} - \frac{t_4}{2})(+\frac{t_1}{2}-\frac{t_2}{2} -\frac{t_3}{2} + \frac{t_4}{2}).
$$

Notice that $p_0(x) = p_+(x) + p_-(x)$.

Using the first two restrictions and the above identity, we get that 

$$
p^M = \alpha^\vee_{1,1}\otimes p_0 - \alpha^\vee_{1.2}\otimes p_-
$$

Hence we have that

$$
e^M
\begin{bmatrix}
    0 & t_1 & \\
    -t_1 & 0 & \\
     & & 0 & t_2 \\
     & & -t_2 & 0 \\
     & & & & 0 & t_3 \\
     & & & &  -t_3 & 0 \\
     & & & & & & 0 & t_4 \\
     & & & & & & -t_4 & 0
\end{bmatrix}
= 
\begin{bmatrix}
    p_0 & 1\\
    & -p_- & 1 \\
    & & -p_+
\end{bmatrix}
$$

\subsection*{Quaternionic Cases}
We use a somewhat different argument for the quaternionic cases. We begin by calculating the case where $G_\bbR = \GL_n(\bbH)$. For this case, consider the coweight $e^1$ with $e^1(t) = \textnormal{diag}(t,1,\dots,1) \in (\bbH^*)^n$ for $t\in \bbC^*$.  We have that $\Gr_{e^1} \cong \bbH P^{n-1}$, with ${M_\bbR} = (\bbH^*)^n$-fixed points corresponding to $e^i$ mapping $\bbC^*$ into the $i$th diagonal coordinate.  We have that ${M_\bbR}$ acts on the fiber of $\calL_{det, \bbH P^{n-1}}$ over $\{t^{e^i}\}$ by the vector representation with weights $\textnormal{diag}(t_1,t_2,\dots, t_m) \mapsto \pm t_i$. Hence the restriction of the equivariant cohomology class to $\{t^{e^i}\}$is $t_i^2$.  Consider the basis $e_i$ of $\Lambda_S$ dual to $e^i$.  With respect to this basis, then, we have that

$$
p^M = \sum_i e_i \otimes t_i^2
$$

In terms of matrices, we may think of $e^M$ as the map $\frakc \rightarrow \gl_n(\bbC)$ given by

$$
e^M
\begin{bmatrix}
    t_1 \\
    & t_2 \\
    & & \ddots \\
    & & & t_n
\end{bmatrix}
=
\begin{bmatrix}
    t_1^2 & 1\\
    & t_2^2 & 1\\
    & & \ddots & \ddots \\
    & & & t_{n-1}^2 & 1 \\
    & & & & t_n^2
\end{bmatrix}
$$

Complete the list $\alpha^\vee_i$ to a basis by appending the trace $\chi^\vee = \sum e_i$.  With respect to this basis, we have the following formula for $p^M$.  Let $\varphi: (\bbC)^n \rightarrow (\bbC^n)$ denote squaring, so that $\omega_1 \circ sq = \frac{1}{n}((n-1)t_1^2-t_2^2 - t_3^2 - \dots - t_n^2)$ and so on.  Then  

$$
p^M = \frac{1}{n}\chi^\vee \otimes (\chi^\vee\circ \varphi) + \sum_i \alpha^\vee_i \otimes (\omega_i\circ \varphi)
$$

Similarly, with respect to the basis $\alpha^\vee_{1,i}$ appended with the trace $\chi^\vee$:

$$
p^M = \frac{1}{n}\chi^\vee \otimes (\chi^\vee\circ \varphi) + \sum_i \alpha^\vee_{1,i} \otimes (\omega_{1,i}\circ \varphi)
$$

$$
e^M
\begin{bmatrix}
    t_1 \\
    & t_2 \\
    & & \ddots \\
    & & & t_n
\end{bmatrix}
=
\begin{bmatrix}
    \omega_{1,1}\circ \varphi & 1\\
    & \omega_{1,2}\circ \varphi & 1\\
    & & \ddots & \ddots \\
    & & & \omega_{1,n-1}\circ \varphi & 1 \\
    & & & & \omega_{1,n}\circ \varphi
\end{bmatrix}
+
\frac{1}{n}\chi I
$$

For $G_\bbR = \SL_n(\bbH)$, we use the restriction of the above class.  Observe that the term $\frac{1}{n}\chi^\vee \otimes (\chi^\vee\circ \varphi)$ vanishes at every $t^\nu$ where $\nu$ is a coweight of $\SL_n(\bbH)$.  Hence, for $\SL_n(\bbH)$, we have that $e^M$ lands in $\pgl_n(\bbC)$:

$$
e^M
\begin{bmatrix}
    t_1 \\
    & t_2 \\
    & & \ddots \\
    & & & t_n
\end{bmatrix}
=
\begin{bmatrix}
    \omega_{1,1}\circ \varphi & 1\\
    & \omega_{1,2}\circ \varphi & 1\\
    & & \ddots & \ddots \\
    & & & \omega_{1,n-1}\circ \varphi & 1 \\
    & & & & \omega_{1,n}\circ \varphi
\end{bmatrix}
$$

For $G_\bbR = \PSL_n(\bbH)$, observe that the map $H*_{\PSL_n(\bbH)}(pt) \cong H^*_{\SL_n(\bbH)}(pt)$ is determinant 2, hence the equivariant portion picks up a coefficient of $\frac{1}{2}$:

$$
e^M
\begin{bmatrix}
    t_1 \\
    & t_2 \\
    & & \ddots \\
    & & & t_n
\end{bmatrix}
=
\begin{bmatrix}
    \frac{1}{2}\omega_{1,1}\circ \varphi & 1\\
    & \frac{1}{2}\omega_{1,2}\circ \varphi & 1\\
    & & \ddots & \ddots \\
    & & & \frac{1}{2}\omega_{1,n-1}\circ \varphi & 1 \\
    & & & & \frac{1}{2}\omega_{1,n}\circ \varphi
\end{bmatrix}
$$

\section*{Proof of isomorphism}

We now calculate the image of $\sigma_{M_\bbR}$ and $\sigma$ and show that they are injective.  

Let $B^\vee_{X, e^M}$ denote the centralizer of $e^M$ in $B^\vee_X \times \Spec R_M$.  Since $\sigma_{M_\bbR}$ commutes with $p_1^M$, the image of $\sigma_{M_\bbR}$, considered as a morphism, lands in $B^\vee_{X, e^M}$

Thus, our morphism becomes

$$
\sigma_{M_\bbR}: \Spec H_*^{M_\bbR}(\Gr_\bbR) \rightarrow B^\vee_{X, e^M}
$$

For the torus ${T_\bbR}$, let $B^\vee_{X, e^T}$ be the centralizer of $e^T = e^M$ in $B^\vee_{X} \times \Spec R_{T_\bbR}$.  In this setting, we have

$$
\sigma_{T_\bbR}: \Spec H_*^{T_\bbR}(\Gr_\bbR) \rightarrow B^\vee_{X, e^{T_\bbR}}.
$$

Passing to the nonequivariant setting, we get

$$
\sigma: \Spec H_*(\Gr_\bbR) \rightarrow B^\vee_{X, e}
$$

where $B^\vee_{X, e}$ is the centralizer of $e$ in $B^\vee_{X}$.

We show that these are, in fact, isomorphisms after inverting a few small primes.  For the non-group cases, it suffices to invert at most the prime $2$.

First, we make a series of observations about the image, mirroring Yun and Zhu's Proposition 3.3 and its proof \cite{yun_integral_2011}.  We will ultimately need it to generalize our results to cases other than centerless groups:

\begin{lemma}
(1) There is an isomorphism $\Spec H_0(\Gr_\bbR) \cong Z^\vee_X$, where $Z^\vee_X$ is the center of $G^\vee_X$.

(2)If $G_\bbR$ is of simply connected type--that is, if $X$ is simply connected--then $\sigma$ factors through $\sigma: \Spec H_*(\Gr_\bbR) \rightarrow U^\vee_X$, the unipotent radical of $B^\vee_X$.

(3) Let $\Gr^0_\bbR$ denote the neutral component of $\Gr^0_\bbR$.  WIn general, $\sigma$ factors as
    
    $$\sigma: \Spec H_*(\Gr_\bbR) \xrightarrow{\sim} \Spec H_0(\Gr_\bbR) \times \Spec \Spec H_*(\Gr^0_{\bbR}) \rightarrow Z^\vee_X\times U^\vee_X$$
\end{lemma}

\begin{proof}
We may identify $\pi_0(\Gr_\bbR) \cong \pi_1(X)$ with $\Lambda_S/\bbZ \Phi_X$. The proof is similar to that for the complex affine Grassmannian, for details, see \cite{loos}.  Hence $H_0(\Gr_\bbR) \cong \Lambda_S/\bbZ \Phi_X \cong \oh(Z^\vee_X)$ as Hopf algebras.

As an automorphism, $\sigma$ descends to an automorphism of the associated graded functor

$$\bigoplus_{\nu\in \Lambda_S} H^*_c(S_{\nu,\bbR},-) \otimes H_*(\Gr_\bbR) \rightarrow \textnormal{Mod}^{\Lambda_S}(H_*(\Gr_\bbR))$$.  

Hence $\sigma$ induces a map $\Spec H_*(\Gr_\bbR) \rightarrow T_X^\vee$ equal to the composition $\Spec H_*(\Gr_\bbR) \rightarrow B_X^\vee \rightarrow T_X^\vee$.

But the graded pieces $$H^*_{T_\bbR}(S_{\nu,\bbR},\calF)$$ are concentrated in one degree.  Hence the automorphism induced by $\sigma$ may be calculated by taking the class corresponding to $\nu \in H_0(\Gr_\bbR)$, expanding it to a basis of $H_*(\Gr_\bbR)$, then taking a dual basis of $H^*(\Gr_\bbR)$ starting with $\nu^*$, which takes the value $1$ on the component of $\nu$ and $0$ elsewhere. Then it is clear that $\sigma$ induces the following automorphism of $H^*_{T_\bbR}(S_{\nu,\bbR},-) \otimes H_*(\Gr_\bbR) \rightarrow \textnormal{Mod}^{\Lambda_S}(H_*(\Gr_\bbR)$:

$$
v \mapsto (\nu^*\cup v)\otimes \nu
$$

Since every $\nu\in H_0(\Gr_\bbR)$, our graded functor factorizes as 

$$\oplus_{\nu\in \Lambda_S} H^*_{T_\bbR}(S_{\nu,\bbR},-) \otimes H_*(\Gr_\bbR) \rightarrow \textnormal{Mod}^{\Lambda_S}(H_0(\Gr_\bbR))\rightarrow \textnormal{Mod}^{\Lambda_S}(H_*(\Gr_\bbR))$$

where the last arrow is extension of scalars to $H_*(\Gr_\bbR)$.  Hence our map $\Spec H_*(\Gr_\bbR) \rightarrow B_X^\vee \rightarrow T_X^\vee$ has image in $Z^\vee_X$.  Taking the preimage, we get that $\Spec H_*(\Gr_\bbR) \rightarrow B_X^\vee$ has image in $Z_X^\vee \times U_X^\vee$.

If $G_\bbR$ has simply connected root data--i.e. if $X$ is simply connected--then clearly $Z^\vee_X = \{e\}$ and thus $\Spec H_*(\Gr_\bbR) \rightarrow B_X^\vee$ has image in $U_X^\vee$

For the last portion, observe that the neutral component $\Gr^0_\bbR$ corresponds to the affine Grassmannian $\Gr^{sc}_\bbR$ of $(G^{sc})_\bbR$.  By functoriality, we have that the map $\Spec H_*(\Gr_\bbR) \rightarrow \Spec H_*(\Gr^{sc}_\bbR)$ is just the projection $Z_X^\vee \times U_X^\vee \rightarrow U_X^\vee$--in other words, we have the last assertion.
\end{proof}

\newtheorem*{t7}{Theorem 7}
\begin{t7}
The map $\sigma$, resp. $\sigma_{T_\bbR}$ and $\sigma_{M_\bbR}$, are closed embeddings.
\end{t7}

\begin{proof}
We prove the non-equivariant case.  The ${T_\bbR}$ and ${M_\bbR}$-equivariant cases follow by lifting basis elements.

For now, assume $G_\bbR$ is centerless.  We wish to show that the image of every $H_*(\Gr_{\nu, \bbR})$ inside $H_*(\Gr_\bbR)$ lies in the image of $\oh(B^\vee_{X, e})$.  If such a statement holds, then the map $\oh(B^\vee_{X, e})\rightarrow H_*(\Gr_\bbR)$ is surjective, since $H_*(\Gr_\bbR)$ is generated by the fundamental class of one cell $\Gr_{\nu, \bbR}$. 

To show that such a cell exists, we appeal to Theorems 1-3 of Bott's paper "The Space of Loops on a Lie Group" \cite{Bott_1958}.   Bott works with the case where the symmetric space is a centerless compact Lie group.  However, the results of Bott-Samelson's prior paper \cite{bottsamelson} give most of the setup for Theorem 1 in our case.  The remainder of the proofs of Theorems 1 and 2 follow by observation that, for centerless, splitting-rank real groups, the root datum of the relative dual group is that of $\SL_n(\bbC)$. Theorem 3 is an application of Theorems 1 and 2, together with homological commutativity of $H^*(\Omega(G/K),\bbQ)$.  As such, Bott's results immediately extend, and we have our generating cell $\Gr_{\nu, \bbR}$.

Now, consider the following sequence of sheaves on $\Gr_{\leq \nu, \bbR}$:

$$
\bbZ \rightarrow \calI^\nu_*[-\langle \rho^\vee, \nu \rangle] \rightarrow j^\nu_*\bbZ
$$

Taking cohomology, we get the maps

$$
H^*(\Gr_{\leq \nu, \bbR}) \rightarrow H^{*-\langle \rho^\vee, \nu \rangle} (\calI^\nu_*) \rightarrow H^*(\Gr_{\nu, \bbR}).
$$

Here the first map is given explicitly by cupping with $v_{low} = [\Gr_{\leq \nu, \bbR}]$.  The composition and the second maps are given by the natural restrictions.

Dualizing, we get maps in homology:

$$
H_*(\Gr_{\nu, \bbR}) \rightarrow H_{*-\langle \rho^\vee, \nu \rangle} (\calI^\nu_*) \rightarrow H_*(\Gr_{\leq \nu, \bbR}) .
$$

Here our map $H_{*-\langle \rho^\vee, \nu \rangle} (\calI^\nu_*) \rightarrow H_*(\Gr_{\leq \nu, \bbR})$ is the adjoint of $H^*(\Gr_{\leq \nu, \bbR}) \rightarrow H^{*-\langle \rho^\vee, \nu \rangle} (\calI^\nu_*)$.  Hence its composition with $H_*(\Gr_{\leq \nu, \bbR}) \rightarrow H_*(\Gr_\bbR)$ is given by

$$
u \longmapsto \sum_i \langle u, h^i\cdot[\Gr_{\leq \nu, \bbR}]\rangle\,h_i.
$$

Now, there is another natural map $H_{*-\langle \rho^\vee, \nu \rangle} (\calI^\nu_*) \rightarrow H_*(\Gr_{\bbR})$.  Recall that $H^{*-\langle \rho^\vee, \nu \rangle} (\calI^\nu_*)$ is the $G^{\vee}_{X}$-Schur module of lowest weight $[\Gr_{\leq \nu, \bbR}]$.  We map $H_{*-\langle \rho^\vee, \nu \rangle} (\calI^\nu_*) \rightarrow \oh(B^\vee_{X,e})$ by the matrix coefficient of $H^{*-\langle \rho^\vee, \nu \rangle} (\calI^\nu_*)$ with respect to $[\Gr_{\leq \nu, \bbR}]$. Explicitly, the map is given by

$$
u \mapsto \langle u, b\cdot [\Gr_{\leq \nu, \bbR}] \rangle
$$

for all $b\in B^\vee_{X,e}$. We now take the composition 

$$
H_{*-\langle \rho^\vee, \nu \rangle} (\calI^\nu_*) \rightarrow \oh(B^\vee_{X,e}) \xrightarrow \sigma^* H_*(\Gr_\bbR).
$$

This composition is given by the following formula:

$$
u \longmapsto \langle u, b\cdot \sigma([\Gr_{\leq \nu, \bbR}]) \rangle = \sum_i \langle u, h^i\cdot[\Gr_{\leq \nu, \bbR}]\rangle\,h_i
$$

This is the same as our earlier map.  In other words, the following square commutes:

$$
\begin{tikzcd}
H_{*-\langle \rho^\vee, \nu \rangle} (\calI^\nu_*)\arrow[d] \arrow[r] & \oh(B^\vee_{X,e})\arrow[d] \\
H_*(\Gr_{\leq \nu, \bbR}) \arrow[r] & H_*(\Gr_\bbR)
\end{tikzcd}
$$

Precomposing our first map $H_*(\Gr_{\nu, \bbR}) \rightarrow H_{*-\langle \rho^\vee, \nu \rangle} (\calI^\nu_*) $, we get that the image of $H_*(\Gr_{\nu, \bbR})$ in $H_*(\Gr_\bbR)$ lies inside the image of $\sigma^*$.  

Allowing $\nu$ to range over all cocharacters, we get that $\sigma^*$ is a surjection.  Hence $\sigma: \Spec H_*(\Gr_\bbR) \rightarrow B^\vee_{X, e}$ is a closed embedding.

To extend the result to splitting-rank $G_\bbR$ with nontrivial center, we apply the result to the centerless quotient $G^{adj}_\bbR$ of $G_\bbR$, whose affine Grassmannian $\Gr^{adj}_\bbR$  has homology $H_0(\Gr^{adj}_\bbR) \otimes H_*(\Gr^{adj,0}_\bbR)$ by Lemma 5.  We have that $\Spec(H_0(\Gr^{adj}_\bbR) \otimes H_*(\Gr^{adj,0}_\bbR)) \rightarrow Z^{\vee}_{X,adj} \times U^{\vee}_{X,adj}$ is a closed embedding.  We then notice that $\Gr^{adj,0}_\bbR \cong \Gr^{0}_\bbR$ and $Z^\vee_{X,adj}$ is finite.  Matching $U^\vee_{X,adj}$ with $U^\vee_{X}$, we get the result for $\Gr^{0}_\bbR$--i.e. that $\Spec H_*(\Gr^0_\bbR) \rightarrow U^\vee_X$ is a closed embedding.  Then we use the isomorphism $\Spec H_0(\Gr_\bbR) \cong Z^\vee_X$ to get that the whole map is a closed embedding.
\end{proof}

For injectivity, we need an intermediate lemma on flatness.  Let $\ell_{G}$ be the squared ratio of the lengths of the real coroots of $G_\bbR$.

\begin{lemma}
The scheme $B^\vee_{X,e}$ is flat over $\bbZ[1/\ell_G]$.

Likewise, $B^\vee_{X,e^T}$ (resp. $B^\vee_{X, e^M}$) is flat over $R_{T_\bbR}[1/n_G\ell_G]$ (respectively $R_{M_\bbR}[1/n_G\ell_G]$).
\end{lemma}

\begin{proof}
The proof is exactly as in [YZ]. First, we prove the non-equivariant case. Then we will reduce both equivariant cases to the non-equivariant case. 

Define the following map:

$$
\varphi: B^{\vee}_{X} \rightarrow \mathfrak{u}^\vee_X \times \Spec\bbZ[1/\ell_G]
$$

given by $\varphi(b) = \Ad(b)e-e$.  By construction, $B^\vee_{X,e}$ is precisely the fiber over zero of this map.  Each fiber of $B^\vee_{X,e} \rightarrow \Spec\bbZ[1/\ell_G]$, then, has dimension at least $r = \dim B^\vee_X - \dim u^\vee_X$.  If each fiber has dimension exactly $r$, then $B^\vee_{X,e}$ is locally a complete intersection over $\Spec\bbZ[1/\ell_G]$.  Since $\Spec\bbZ[1/\ell_G]$ is regular, then $B^\vee_{X,e}$ will be flat.

We check the dimension of the fiber over each prime $p$ of $\Spec\bbZ[1/\ell_G]$.  Base change $B^\vee_X$ to $k = \overline{\bbF}_p$.  In this setting, we may scale $e = \sum_i n_ix_i$ by an element of $T(k)$ to the standard principal nilpotent $e_0 = \sum_i x_i$.  Then $B^\vee_{X,e} = Z^\vee \times U^\vee_{X,e_0}$.  Since $\dim U^\vee_{X,e_0} = r-\dim Z^\vee$, we get that our fiber dimension is exactly $r$.  Hence $B^\vee_{X,e}$ is flat over $\bbZ[1/\ell_G]$.

For the equivariant cases, we reduce to the nonequivariant case in the following way.  We treat the case of $R_{T_\bbR}[1/\ell_Gn_G]$--an identical argument applies to $R_{M_\bbR}[1/\ell_Gn_G]$.  

We mimic the construction of $\varphi$ in this setting.  Consider the morphism

$$
\Ad(-)e^M-e^M: B^\vee_X \times \Spec R_{T_\bbR}[1/\ell_Gn_G] \rightarrow \mathfrak{u}^\vee_X \times \Spec R_{T_\bbR}[1/\ell_Gn_G].
$$

Like in the non-equivariant case, $B^\vee_{X, e^T}$ is the fiber above zero of this map.  As before, $B^\vee_{X, e^T}$ will be locally a complete intersection, and hence flat, if every fiber dimension is exactly $r$.  

Recall that $\Spec R_{T_\bbR}$ comes equipped with a natural $\bbG_m$-action given by the grading on $R_{T_\bbR}$, along with the inclusion $\Spec \bbZ \rightarrow \Spec R_{T_\bbR}$ given by letting every monomial tend to zero.  The $\bbG_m$-action is compatible with all base changes $\Spec (R_{T_\bbR}\otimes k)$ for $k$ the residue field of a prime of $\bbZ$.  Hence, for a prime $P \in \Spec R_{T_\bbR}$ lying above $p \in \Spec \bbZ$, we have that the $\bbG_m$-orbit of $P$ remains above $p$.  The closure of the orbit, then, contains $p$.  Thus we have $\dim B^\vee_{X, e^T, P} \leq \dim B^\vee_{X, e^T, p}$, where $p\in \bbZ$.  Since $B^\vee_{X, e^T, p}$ is simply the fiber of $B^\vee_{X, e}$ over $p$, its fiber dimensions are all $r$.  Thus, every fiber in the equivariant case also has dimension $r$, proving the assertion.
\end{proof}

Now we prove injectivity and hence prove isomorphism.

\newtheorem*{t8}{Theorem 8}
\begin{t8}
We have the following isomorphisms:
\begin{align}
    \Spec H_*(\Gr_{\bbR}, \bbZ[1/\ell_G]) &\cong B^\vee_{X,e}[1/\ell_G] \\
    \Spec H_*^{T_\bbR}(\Gr_{\bbR}, \bbZ[1/n_G\ell_G]) &\cong B^\vee_{X,e^T}[1/\ell_G]\\
    \Spec H_*^{M_\bbR}(\Gr_{\bbR}, \bbZ[1/n_G\ell_G]) &\cong B^\vee_{X,e^M} [1/\ell_G]
\end{align}
\end{t8}

\begin{proof}
We begin with the ${T_\bbR}$-equivariant case, where we can use the equivariant localization theorem.

Up to contracting our semi-infinite orbits $S_{\nu,\bbR}$, our fixed points are given by $t^\nu$ with $\nu \in X_*(A)$ for $A$ the maximal split torus.  Using the localization theorem, the natural map

$$
Loc^*: H_{T_\bbR}^*(\Gr_\bbR) \rightarrow \prod_{\nu\in \Lambda_S} H_{T_\bbR}^*(\{t^\nu\}). 
$$

becomes an isomorphism when tensored with $Q = \textnormal{Frac } R_{T_\bbR}$.  

We take the adjoint:

$$
Loc_*: \bigoplus_{\nu\in \Lambda_S} H_{T_\bbR}^*(\{t^\nu\}) \rightarrow H^{T_\bbR}_*(\Gr_\bbR) 
$$

Identify the sum on the left with $R_{T_\bbR}[\Lambda_S]$ and take the spectrum.  We get the following map:

$$
\Spec(Loc_*): \Spec H^{T_\bbR}_*(\Gr_\bbR)  \rightarrow T^\vee \times \Spec R_{T_\bbR} 
$$

This becomes an isomorphism after base-changed with $\Spec Q$. 

We claim that this factors as follows:

$$
\begin{tikzcd}
\Spec (H^{T_\bbR}_*(\Gr_\bbR)) \arrow[r, "\sigma_{T_\bbR}"] \arrow[rd] & \arrow[d] B^\vee_{X, e^T}[1/\ell_G]\\
& T^\vee \times \Spec R_{T_\bbR}
\end{tikzcd}
$$

Consider the composition of $\sigma_{T_\bbR}$ with projection to $T^\vee \times \Spec R_{T_\bbR}$.  Observe that the action of $H_{T_\bbR}^*(\Gr_\bbR)$ on each character functor $H_{{T_\bbR},c}(S_{\nu,\bbR}, \calF)$ factors through restriction to $H_{T_\bbR}(S_{\nu,\bbR}) = H_{T_\bbR}(\{t^\nu\})$.  By adjunction, for $i_\nu = \{t^\nu\} \rightarrow \Gr_\bbR$ and $v_\nu \in H_{{T_\bbR},c}(S_{\nu,\bbR}, \calF)$, we have that

$$
\sigma_{T_\bbR}(v_\nu) = \sum_{\nu} (i_nu^*(h^i)\cup v_\nu) \otimes h_i = v_\nu\otimes i_{\nu,*}(1).
$$

Summing over all $i_{\nu,*}$, we get the map $Loc_*$.  Hence the diagram commutes.

Now, after localization, we have that $\Spec (Loc)_*$ is an isomorphism. By principality of $e$, the projection onto $T^\vee \times \Spec R_{T_\bbR}$ is also an isomorphism after localization.  Thus $\sigma_{T_\bbR}$ is an isomorphism after localization.

Using flatness of both $B^\vee_{X,e^T}[1/\ell_G]$ and $\Spec (H^{T_\bbR}_*(\Gr_\bbR, \bbZ[1/n_G\ell_G))$, we get that $\sigma_{T_\bbR}$ is an isomorphism.  This proves the ${T_\bbR}$-equivariant case.

For the ${M_\bbR}$-equivariant case, take the quotient by $W_M$ of both sides.

For the de-equivariantized case, base changing the ${T_\bbR}$-equivariant case to $\Spec \bbQ$ gives an isomorphism $\Spec H_*(\Gr_{\bbR}, \bbQ) \cong B^\vee_{X,e}\times \Spec \bbQ$.  Using flatness, we get the result over $\bbZ[1/\ell_G]$.
\end{proof}

\section*{$G_\bbR$-equivariant homology}

We use the description of ${T_\bbR}$ and ${M_\bbR}$-equivariant homology to calculate the $G_\bbR$-equivariant homology of $\Gr_\bbR$.  Recall that $G_\bbR(\bbR[[t]])$ retracts onto $G_\bbR$, which in turn retracts onto its maximal compact subgroup $K$.

After inverting 2, we have that $R_K = (R_{T_\bbR})^{W_K}$, where $W_K$ is the Weyl group of $K$.  Thus we may identify

$$
H_*^{G_\bbR}(\Gr_\bbR) \cong (H_*^{T_\bbR}(\Gr_\bbR))^{W_K} \cong \oh(B^\vee_{X, e^T})^{W_K}.
$$

To calculate the quotient $B^\vee_{X, e^{T_\bbR}}/W_K$, we filter $W_K$ in two steps.  We know that $B^\vee_{X, e^M} = B^\vee_{X, e^T}/W_K$.  Since $W_K$ is normal with $W/W_K \cong W(G^\vee_X)$, we note that $B^\vee_{X, e^T}/W_K \cong B^\vee_{X, e^M}/W(G^\vee_X)$.  By \cite{yun_integral_2011}, after inverting the set $S_2$ of all bad primes of $G^\vee_{X}$ and those dividing $n$, we may identify $B^\vee_{X, e^M}/W(G^\vee_X)$ with the regular centralizer group scheme $J^\vee_X$ of $G^\vee_X$.

Hence, we have the following theorem.

\newtheorem*{t9}{Theorem 9}
\begin{t9}
For $G_\bbR$ of splitting rank, there is an isomorphism of group schemes:

$$
\Spec H_*^{G_\bbR}(\Gr_\bbR, \bbZ_{S_2}) \cong J^\vee_X \times \Spec \bbZ_{S_2}
$$
\end{t9}

We conclude with a quick observation about the other cases.  We have only used $G_\bbR$ being centerless to guarantee a minuscule cocharacter $\nu$--hence a smooth variety $\Gr_{\leq \nu,\bbR}$ generating $H_*(\Gr_\bbR,\bbQ)$.  Other forms may have such a coweight:  in particular, by a paper of Crabb and Mitchell, the form $G_\bbR = \GL_n(\bbH)$ corresponding to $G/K = \textnormal{U}(2n)/\Sp(n)$ has such a coweight \cite{crabb_mitchell_1988}.  For these our results hold by an identical proof.  Hence, we have the following corollary:

\newtheorem*{c1}{Corollary 1}
\begin{c1}
For $G_\bbR$ with a minuscule generating cocharacter, we have the following isomorphisms:

\begin{align*}
    \Spec H_*(\Gr_{\bbR}, \bbZ[1/\ell_G]) &\cong B^\vee_{X,e}[1/\ell_G] \\
    \Spec H_*^{M_\bbR}(\Gr_{\bbR}, \bbZ[1/n_G\ell_G]) &\cong B^\vee_{X,e^M}[1/\ell_G]\\
    \Spec H_*^{G_\bbR}(\Gr_\bbR, \bbZ_{S_2}) &\cong J^\vee_X \times \Spec \bbZ_{S_2}
\end{align*}
\end{c1}

\section*{Acknowledgements}
I would like to acknowledge my advisor Tsao-Hsien Chen for his help with this work, as well as David Nadler and Mark Macerato for insightful conversations.

This research is supported by NSF grant DMS 2143722.

\printbibliography
\end{document}